\documentclass{article}
\usepackage{amsmath,amssymb,amsthm}
\usepackage{amscd}
\usepackage{graphicx}
\newtheorem{theorem}{Theorem}
\newtheorem{definition}{Definition}
\newtheorem{remark}{Remark}

\title{The nontrivial kernel of Manturov-Nikonov map from classical braids to virtual braids}
\author{Yangzhou Liu}
\date{March 2026}

\begin{document}
\maketitle
\begin{abstract}
In 2022, V. O. Manturov and I. M. Nikonov \cite{Man22} constructed two composite maps, one of them is following:
$$\begin{CD} 
PB_{n+1} @>p_k>> CPB_{n} @>f_d>> VCB_{n}@>\rho>> GL_n(\mathbb{Z}[t^{\pm 1},s^{\pm 1}])
\end{CD}$$

In this paper, we prove that M-N is unfaithful if $k\geq 6$ since Burau kernel is a subgroup of M-N kernel.

\end{abstract}



\section{Introduction}
As we know, the representation theory of classical braids is more established. For example, we have the Burau representation \cite{Big99}, which is not faithful if $n\geq 5$ for $B_n$, and $n=4$ are still open; and the Lawrence-Krammer-Bigelow representation (which is faithful for $n\geq 1$!\cite{Big02}). Meanwhile, unlike classical knots, virtual knots have many unique properties, such as the parity \cite{Man10}. Thus, if we can map classical braids into virtual braids, we can get some new properties. This is the starting point for the M-N map. The paper \cite{Man22} provides an example of an element in the kernel of the Burau representation that is detected by the M-N map. In contrast, this paper demonstrates that for sufficiently large $k$ (specifically, $k \geq 6$), the Burau kernel becomes a subgroup of the kernel of the M-N map.

\section{Preliminaries and Notations}
Let $\beta$ be a pure braid on $n$ strands. Then $\beta=\{\beta_k\}_{k=1}^n$ where $\beta_k: [0,1]\rightarrow \mathbb{R}^2\cong \mathbb{C}$ are the strands of the braid.

\begin{definition}
Fix an index $k\in\{1,\cdots,n\}$. The maps $p_k(\beta)_l: [0,1]\rightarrow\mathbb{S}^1 = \{z \in \mathbb{C} \mid |z| = 1\}$,$ l\neq  k$, given by the formulas 
$$p_k:\beta_l \mapsto \cfrac{\beta_l(t)-\beta_k(t)}{|\beta_l(t)-\beta_k(t)|},\quad l\neq k.$$
thus we have (Proposition 1 in \cite{Man22}) 
\begin{equation}
            \sigma_i^\epsilon\mapsto 
        \begin{cases}
             \sigma_{k-i-1}^\epsilon,& i\neq k,k-1; \\
             \zeta^{-1},& i= k-1,\epsilon=1; \\
             \Delta_c,& i= k-1,\epsilon=-1; \\
             \Delta_c^{-1},& i= k,\epsilon=1; \\
            \zeta,& i= k,\epsilon=-1.
        \end{cases} 
        \end{equation}
where $\Delta_c=\sigma_1\cdots\sigma_{n - 1}$.
\end{definition}

\begin{figure}
    \centering
    \includegraphics[width=1\linewidth]{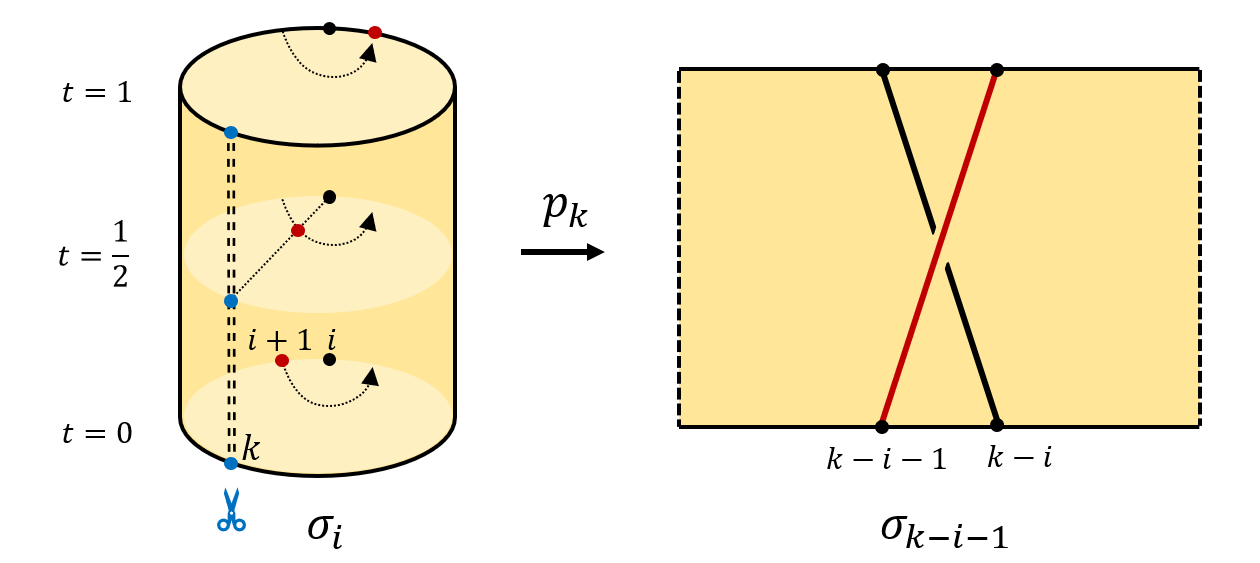}
    \caption{$p_k$}
    \label{fig:1}
\end{figure}
Let $\beta'$ be a pure braid diagram in the cylinder on $n'= n-1$ strands. Assume that the initial configuration of the braid $\{\beta_k'(0)\}\subset \mathbb{S}^1$ is in general position (here this means that $\frac{\beta_k'(0)}{\beta_l'(0)}, k\neq l$, are transcendental numbers). 
\begin{definition}
Let $d\in \mathbb{N}$, consider the map $f_d(z)=z^d$. This geometric construction induces a homomorphism on the level of braid groups, whose action on the standard Artin generators is given by the following formula (Proposition 1 in [Man22]):
 \begin{equation}\notag
    \begin{array}{cccc}
    f_d:&\sigma_i &\longmapsto& \sigma_i \\
        &\zeta &\longmapsto& \zeta(\Delta_v \zeta)^{d-1}
    \end{array} 
\end{equation}
where $\Delta_v=\tau_1\cdots \tau_{n-1}$, and $\tau_i$ is virtual crossings generators.  
\end{definition}

\begin{figure}
    \centering
    \includegraphics[width=1\linewidth]{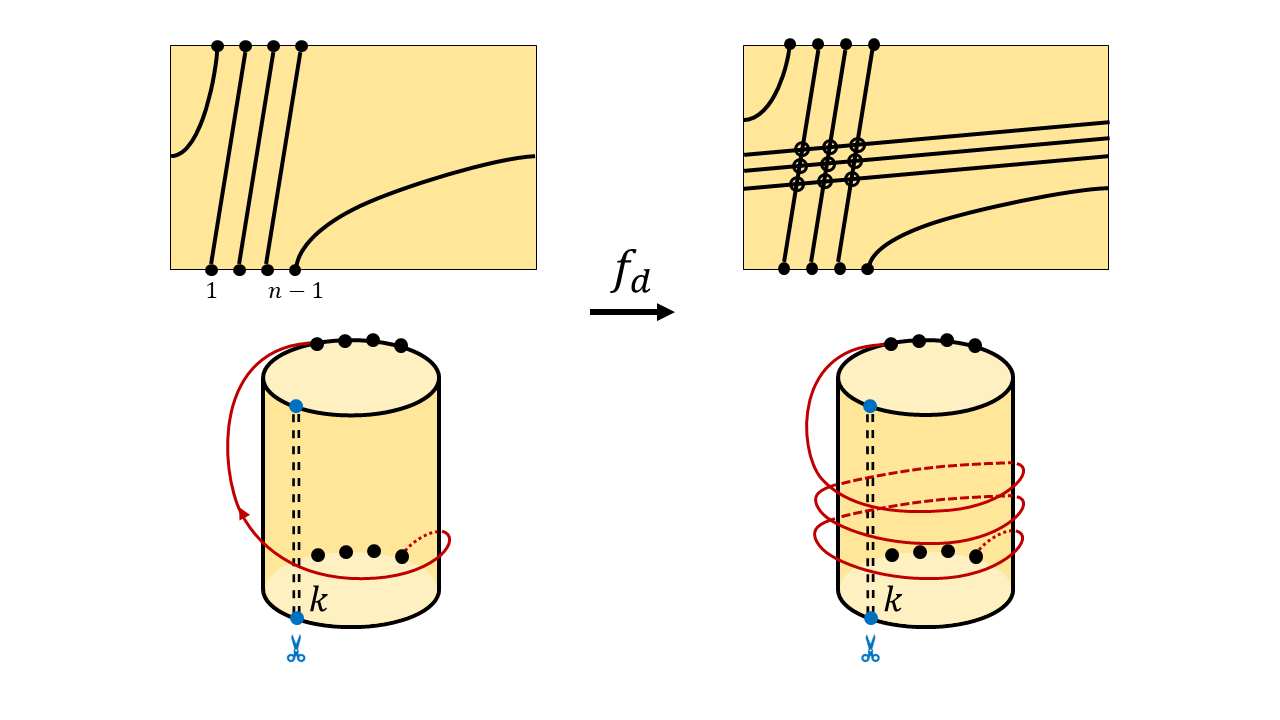}
    \caption{$f_d$}
    \label{fig:2}
\end{figure}

\begin{definition}
The map $\rho:VCB_n\rightarrow GL(n,\mathbb{Z}[t^{\pm 1},s^{\pm 1}])$, 
$$\rho:
         \sigma_k\mapsto \begin{pmatrix}
            I_{k-1}\\
            &1-t& t\\
            &1&0\\
            &&& I_{n-k-1}
         \end{pmatrix},\qquad 
         \tau_k\mapsto \begin{pmatrix}
         I_{k-1}\\
           &0&s\\
           &s^{-1}&0\\
           &&& I_{n-k-1}
         \end{pmatrix},$$ 
        $$\zeta\mapsto
         \begin{pmatrix}
            {\textbf{0}}&I_{n-1}\\
            1& {\textbf{0}}
         \end{pmatrix}.$$
\end{definition}
\begin{remark}
Let $\psi: B_n\rightarrow GL(n,\mathbb{Z}[n,t^{\pm 1} ])$ denotes Burau representation, thus $\rho|_{B_n}=\psi$. 
\end{remark}

\section{Unfaithfulness of M-N map}
\begin{theorem}
Manturov-Nikonov map $\rho\circ f_d\circ p_k$ is unfaithful for $k\geq 6$ and any $d\in \mathbb{Z}$.
\end{theorem}
\begin{proof}
As we know, the Burau representation is unfaithful for $n\geq 5$\cite{Big99}, for exmaple, the following element lies in the kernel in the Burau representation $\psi$,
$$\alpha:=[\gamma\sigma_4\gamma^{-1},\sigma_4\sigma_3\sigma_2\sigma_1^2\sigma_2\sigma_3\sigma_4]$$
where $\gamma\in B_5$ thus $\alpha$ so does. 

We can pick $p_6:PB_6\rightarrow CPB_5$, we have $$p_k(\sigma_i)=\sigma_{k-i-1}, i\neq k,k-1$$
i.e.
$$\sigma_4\mapsto \sigma_1,\quad\sigma_3\mapsto \sigma_2$$
$$\sigma_2\mapsto \sigma_3,\quad\sigma_1\mapsto \sigma_4$$
thus there exist a $\tilde{\alpha}$, satisfying
$$p_6(\tilde{\alpha})=\alpha$$
then we have 
$$\rho \circ f_d \circ p_6(\tilde{\alpha})=\rho \circ f_d(\alpha)=\rho (\alpha)=\psi(\alpha)=I$$
since $f_d|_ {B_5}=id|_ {B_5}$,except for $f_d(\zeta)=\zeta(\Delta_v\zeta)^{d-1}$, but we avoid that $\zeta$ appears in $\alpha$.

Therefore, $\tilde{\alpha}\in \ker \rho\circ f_d\circ p_6$ and it completes the proof. 
\end{proof}

\begin{theorem}
Let $n = 2m$ be an even integer and $d = 1$. Then the M-N map $\rho \circ f_1 \circ p_k$ is not faithful.
\end{theorem}
\begin{proof}
Consider the non-trivial pure braid $\beta = \sigma_k^{-2m} \in PB_{n+1}$. Since $p_k(\sigma_k^{-1}) = \zeta$, we have $p_k(\beta) = \zeta^{2m}$. By definition, $f_1(\zeta^{2m}) = \zeta^{2m}$. Finally, $\rho(\zeta)$ is a permutation matrix of order $n$, hence $\rho(\zeta^{2m}) = \rho(\zeta^n) = I$. Therefore, $\beta \in \ker(\rho \circ f_1 \circ p_k)$.
\end{proof}

\section*{Acknowledgments}
I extremely grateful to Vassily O. Manturov and  Igor M. Nikonov.

\end{document}